\documentclass[11pt]{article}

\usepackage{amssymb}
\usepackage{color, amsmath,amssymb, amsfonts, amstext,amsthm, latexsym}

\usepackage{epsfig, graphicx, graphics}
\usepackage{longtable}

\numberwithin{equation}{section}

\renewcommand{\phi}{\varphi}

\newcommand{\eps}{\varepsilon}

\newtheorem{definition}{Definition}[section]
\newtheorem{theorem}{Theorem}[section]
\newtheorem{lemma}{Lemma}[section]
\newtheorem{remark}{Remark}[section]

\title{Effective dynamics\\ of a coupled microscopic-macroscopic  stochastic system
\footnote{This work was partly supported by   NSF of China grants
(No. 10901065, No. 10971225, and No. 11028102), the NSF Grants  1025422 and
0731201,
 and the Cheung Kong Scholars Program.
  } }

\author{Jian Ren$^1$, Hongbo Fu$^1$, Daomin Cao$^2$ and Jinqiao Duan$^{3, 1}$      \\
\\
1. School of Mathematics and Statistics\\Huazhong University of Science and Technology \\
  Wuhan 430074, China\\\emph{E-mail: renjian0371@163.com} \\
2. Institute of Applied Mathematics\\Chinese Academy of
Sciences\\Beijing 100190, China\\
3. Department of Applied Mathematics\\ Illinois Institute of Technology \\
 Chicago, IL 60616, USA \\
\emph{E-mail: duan@iit.edu}   \\
     }

\begin{document}
\date{\today}

\maketitle

\pagestyle{plain}

\begin{abstract}  A conceptual model for microscopic-macroscopic slow-fast stochastic systems is considered.
A dynamical reduction procedure is presented in order to extract effective dynamics for this kind of systems.  Under
appropriate assumptions, the effective system is shown to approximate the original system, in the sense of a probabilistic convergence.

\medskip





\end{abstract}

\section{Motivation}  \label{intro}

 In modeling complex phenomena in biomedical, geophysical,   and chemical systems, we sometimes encounter      microscopic-macroscopic stochastic systems. These are systems of coupled stochastic ordinary and partial differential equations (SDEs and SPDEs). The SPDEs describe the macroscopic dynamics while SDEs for the microscopic dynamics. For example,   angiogenesis is a vital process in human tissue growth and wound healing. This process involves   the growth of new blood vessels from pre-existing vessels where blood cells    penetrate into growing tissue, supplying nutrients and oxygen and removing waste products \cite{Clauss}. During the process, blood cells interact with the tissue mass randomly. Here the blood cells may be regarded as ``particles" while tissue may be described by a ``density" quantity.  We consider a conceptual microscopic-macroscopic stochastic system where the microscopic component is composed of finite number of ``particles" and the macroscopic component is about ``densities" evolution of a finite number of substances. Both particles and substances are interacting randomly or are interacting in a random environment.

More specifically, and for simplicity, we assume that there are only two particles whose positions are at $\xi(t)$ and $\eta(t)$, and assume also that there are two substances with densities $u(x, t)$ and $v(x, t)$, respectively. Here both $\xi$ and $\eta$ satisfy a system of SDEs, and $u$ and $v$ are described by a system of SPDEs. The SDEs and SPDEs are coupled, due to the impact of particles on density evolution. Furthermore, we suppose that $\xi$ and $u$ evolve slowly, but $\eta$ and $v$ progress much faster. We are interested in deriving an effective model for this coupled stochastic system, hopefully involve   only slow variables $\xi$ and $u$.

First, in \S \ref{model}, we   consider a simpler coupled microscopic-macroscopic system when the fast density $v$ is absent and no external noise acting directly on $u$:
\begin{eqnarray}
\begin{cases}
  u_t = u_{xx} + f(u, \xi), \quad u(x, 0)=u_0(x),  \\
 \dot{\xi} = b(\xi, \eta) + \sigma_3(\xi) \dot W_t, \quad \xi(0)=x_0 ,  \\
  \dot{\eta} = \eps^{-1} B(\xi, \eta) + \eps^{-\frac12} \sigma_4(\xi, \eta) \dot W_t, \quad
  \eta(0)=y_0,\\
  u(0, t)=u(1, t)=0,
\end{cases}
\end{eqnarray}
where $t\in [0, T]$, $x\in [0, 1]$,   $\eps $ is a small positive parameter, and
$W_t$ is a standard scalar Brownian motion. The
coefficients $f, b, B, \sigma_3, \sigma_4$ all satisfy Lipschitz and boundedness assumptions. In this
part, $\xi$ is slow component and $\eta$ fast component.
We derive an effective model involves $u$ and $\xi$ only. This result is summarized in \textbf{Theorem \ref{effective1}}.

Then in \S \ref{model2}, we consider a more complex coupled microscopic-macroscopic
stochastic system
\begin{eqnarray}
\begin{cases}
u_t = u_{xx} + f(u,v, \xi) + \sigma_1(u)\dot W_t^1, \quad u(x, 0)=u_0(x),  \\
v_t=\frac{1}{\varepsilon}(v_{xx}+g(u,v,\xi))+\frac{1}{\sqrt{\varepsilon}}
\sigma_2(u,v)\dot W_t^2, \quad v(x, 0)=v_0(x),\\
\dot{\xi} = b(\xi, \eta) + \sigma_3(\xi) \dot W_t^3, \quad \xi(0)=x_0,   \\
\dot{\eta} = \eps^{-1} B(\xi, \eta) + \eps^{-\frac12} \sigma_4(\xi,
\eta)\dot W_t^3, \quad \eta(0)=y_0,\\
u(0, t)= u(1, t)=0,\quad v(0, t)=v(1, t)=0,
\end{cases}
\end{eqnarray}
where $t\in[0, T]$, $x\in [0, 1]$,  $\eps$ is a small positive parameter, and
$\{W_t^i\}_{t\geq 0}, i=1, 2, 3$ are independent scalar
Brownian motions. The coefficients   $f, g, b, B, \sigma_i's$  satisfy some assumptions. In this
setting, $\xi$ and $u$ are slow components while $\eta$ and $v$ fast
components.  The first two equations are macroscopic components coupled
with   the latter two equations for the microscopic components.
We derive an effective model involving only $u$ and $\xi$,  and the result is stated in \textbf{Theorem \ref{Theo}}.

\section{A stochastic microscopic-macroscopic model}    \label{model}

First we consider the following coupled SPDE-SDE system
\begin{eqnarray}\label{coupled}
\begin{cases}
  u_t = u_{xx} + f(u, \xi), \quad u(x, 0)=u_0(x),  \\
 \dot{\xi} = b(\xi, \eta) + \sigma_3(\xi) \dot{W}_t, \quad \xi(0)=x_0 ,  \\
  \dot{\eta} = \eps^{-1} B(\xi, \eta) + \eps^{-\frac12} \sigma_4(\xi, \eta) \dot{W}_t, \quad
  \eta(0)=y_0,\\
  u(0, t)=u(1, t)=0,
\end{cases}
\end{eqnarray}
for $t\in [0, T]$, $x\in [0, 1]$,  where $\eps $ is a small positive
parameter, $W$ is a standard scalar Brownian motion, the coupling
term $f$ :$\mathbb{R}\times\mathbb{R}\rightarrow \mathbb{R} $
satisfies the condition: Globally Lipschitz in $u$ and $\xi$ with
Lipschitz constant $K^f$, and $b(\xi, \eta)$, $B(\xi, \eta)$,
$\sigma_3(\xi)$, $\sigma_4(\xi, \eta)$ are all globally Lipschitz and
bounded.

As in \cite{FW}, p.268, the slow-fast SDEs for $(\xi, \eta)$ above have the following averaged effective dynamical description.

We introduce a random process $\eta^{\xi}(t)$,      defined by the
stochastic differential equation for fixed $\xi \in \mathbb{R}$,
\begin{eqnarray}
\dot \eta^{\xi}(t)=B(\xi, \eta^{\xi}(t))+\sigma_4(\xi,
\eta^{\xi}(t)) \dot W_t, \quad \quad \eta^{\xi}(0)=y_0.
\end{eqnarray}
For any $t\geq 0$ and any $\xi\in \mathbb{R}$,
we  assume that there exists a function $\bar b(\xi)$,
such that  \\
$$\mathbb{E}\big|\frac{1}{T}\int_t^{t+T}b(\xi, \eta^{\xi}(s))ds-\bar b(\xi)\big|<\chi(T),$$
where the non-negative upper bound function $\chi(T)\rightarrow 0$ as $T\rightarrow \infty.$
Then there is an averaged effective model
\begin{eqnarray} \label{averaged}
 \dot{\bar{\xi}}(t) &=& \bar{b}(\bar\xi(t)) + \sigma_3(\bar{\xi}(t)) \dot{W}_t,
 \quad \bar{\xi}(0)=x_0.
\end{eqnarray}
It follows from \cite{FW} that  $\sup\limits_{0\leq t\leq
T}\mathbb{E}|\xi-\bar\xi|^2\rightarrow 0$  and
$\lim\limits_{\varepsilon\rightarrow
0}\mathbb{P}\big\{\sup\limits_{t\in[0,T]}|\xi-\bar\xi|>\delta\big\}=0.$\\

 Now we consider the following effective system for the original microscopic-macroscopic system \eqref{coupled}:
\begin{eqnarray}\label{effective}
\begin{cases}
  \bar{u}_t = \bar{u}_{xx} + f(\bar{u}, \bar{\xi}), \quad \bar{u}(x, 0)=u_0(x),  \\
 \dot{\bar{\xi}} = \bar{b}(\bar\xi) + \sigma_3(\bar{\xi}) \dot{W}_t, \quad
 \bar{\xi}(0)=x_0.
\end{cases}
\end{eqnarray}
We have already known above that $\xi$ converges to $\bar{\xi}$ in probability, uniformly on bounded time intervals.
Our goal in this section is to show that $u$ converges to $\bar{u}$ in some probabilistic sense.

\bigskip

\begin{theorem}(Effective dynamical reduction I) \label{effective1}\\
Under the above assumptions on the coefficients, the system \eqref{effective} is an effective description of the  original system \eqref{coupled}.
That is,     for any $T>0$ and $\delta>0$,
\begin{eqnarray*}
\mathbb{P}\Big\{\sup\limits_{t\in[0,T]}\|u-\bar
u\|^2>\delta\Big\}\rightarrow 0, \quad  \quad as \quad
\varepsilon\rightarrow 0.
\end{eqnarray*}
This says that $u$ converges to  $\bar{u}$ in probability, uniformly on any finite time intervals.
\end{theorem}

\begin{proof}
Recall the Gronwall's inequality in differential
form: Let $z:[0, T]\rightarrow \mathbb{R}$ satisfy the differential
inequality
\begin{eqnarray*}
\frac{d z}{dt} \leq g(t)z+h(t).
\end{eqnarray*}
Then
\begin{eqnarray*}
z(t)\leq z(0) \exp \big(\int_0^t g(r)\,dr\big)+\int_0^t \exp
\big(\int_s^t g(r)\,dr\big)h(s)\,ds.
\end{eqnarray*}

Denoting $U=u-\bar u$,   then
\begin{eqnarray}
U_t &=& U_{xx} + f(u,\xi)-f(\bar u,\bar\xi),
\end{eqnarray}
Multiply each side of the equation above by $2U$ and taking integral,
by Young's inequality and the global Lipschitz condition on $f$, we get
\begin{eqnarray*}
\frac{d}{dt}\|U\|^2 &=& -2\|U_x\|^2 + \|U\|^2  + K^f \|U\|^2 + K^f
|\xi-\bar\xi|^2\\
{}&\leq& (1+K^f)\|U\|^2 + K^f |\xi-\bar\xi|^2.
\end{eqnarray*}
Taking expectation and by the Gronwall's inequality, we obtain
\begin{eqnarray*}
\|U\|^2&\leq& \int_0^t e^{-(1+K^f)(s-t)}K^f |\xi-\bar\xi|^2\,ds\\
{}&\leq& e^{(1+K^f)T}K^f\int_0^T |\xi-\bar\xi|^2\,ds.
\end{eqnarray*}
Thus,
\begin{eqnarray*}
\mathbb{P}\Big\{\sup\limits_{t\in[0,T]}\|U\|^2>\delta\Big\}
&\leq&\mathbb{P}\Big\{\sup\limits_{t\in[0,T]}e^{(1+K^f)T}K^f\int_0^T
|\xi-\bar\xi|^2\,ds>\delta\Big\}\\
{}&\leq&\mathbb{P}\Big\{\int_0^T\sup\limits_{s\in[0,T]}
|\xi-\bar\xi|^2\,ds=T\sup\limits_{t\in[0,T]}|\xi-\bar\xi|^2>\delta/(e^{(1+K^f)T}K^f)\Big\}\\
{}&=&\mathbb{P}\Big\{\sup\limits_{t\in[0,T]}|\xi-\bar\xi|^2>\delta/(Te^{(1+K^f)T}K^f)\Big\}.
\end{eqnarray*}
By the result $\lim\limits_{\varepsilon\rightarrow
0}\mathbb{P}\big\{\sup\limits_{t\in[0,T]}|\xi-\bar\xi|>\delta\big\}=0$
  in \cite{FW}, we finally have
$\mathbb{P}\Big\{\sup\limits_{t\in[0,T]}\|U\|^2>\delta\Big\}\rightarrow
0$ as $\varepsilon\rightarrow 0$.
This completes the proof.
\end{proof}

\section{A more complex stochastic microscopic-macroscopic model}
\label{model2}

In this section  we  consider teh following more complicated slow-fast microscopic-macroscopic stochastic system
\begin{eqnarray}\label{macro-micro}
\begin{cases}
u_t = u_{xx} + f(u,v, \xi) + \sigma_1(u)\dot W_t^1, \quad u(x, 0)=u_0(x),  \\
v_t=\frac{1}{\varepsilon}(v_{xx}+g(u,v,\xi))+\frac{1}{\sqrt{\varepsilon}}
\sigma_2(u,v)\dot W_t^2, \quad v(x, 0)=v_0(x),\\
\dot{\xi} = b(\xi, \eta) + \sigma_3(\xi) \dot W_t^3, \quad \xi(0)=x_0,   \\
\dot{\eta} = \eps^{-1} B(\xi, \eta) + \eps^{-\frac12} \sigma_4(\xi,
\eta)\dot W_t^3, \quad \eta(0)=y_0,\\
u(0, t)= u(1, t)=0,\quad v(0, t)=v(1, t)=0,
\end{cases}
\end{eqnarray}
for $t\in[0, T]$, $x\in [0, 1]$, where $\eps$ is a small positive parameter, and
$\{W_t^i\}_{t\geq 0}, i=1, 2, 3$ are independent scalar
Brownian motions. For the coefficients we have the following assumptions:\\
\textbf{H1}: The drift coefficients $f(u,v, \xi)$
:$\mathbb{R}\times\mathbb{R}\times\mathbb{R}\rightarrow \mathbb{R}$,
 diffusion coefficients $\sigma_1(u)$
:$\mathbb{R}\rightarrow \mathbb{R}$ are Lipschitz continuous with
respect to all three variables and hence also linear growth, i.e.
there exist constant $K_f$, $K_{\sigma_1}$ such that for any $u_1$,
$u$, $v_1$, $v\in H$ and $\xi_1$, $\xi\in\mathbb{R}$,
\begin{eqnarray}
\nonumber|f(u_1,v_1, \xi_1)-f(u, v, \xi)|^2&\leq&
K_f(|u_1-u|^2+|v_1-v|^2+|\xi_1-\xi|^2),\\\nonumber |f(u,
v,\xi)|^2&\leq&K_f(1+|u|^2+|v|^2+|\xi|^2),\\\nonumber
|\sigma_1(u_1)-\sigma_1(u)|^2&\leq&K_{\sigma_1}|u_1-u|^2,\\
|\sigma_1(u)|^2&\leq&K_{\sigma_1}(1+|u|^2).
\end{eqnarray}
In addition, $f$ is bounded, i.e. exists $C_f$, such that
\begin{eqnarray}
|f(u, v, \xi)|\leq C_f.
\end{eqnarray}
\textbf{H2}: There exist constant $K_g$, $K_{\sigma_2}$, such that for any
$u_1$, $u$, $v_1$, $v$, $\xi_1$, $\xi$,
\begin{eqnarray}
\nonumber|g(u_1,v_1, \xi_1)-g(u, v, \xi)|^2&\leq&
K_g(|u_1-u|^2+|v_1-v|^2+|\xi_1-\xi|^2),\\\nonumber |g(u, v,
\xi)|^2&\leq&K_g(1+|u|^2+|v|^2+|\xi|^2),\\\nonumber |\sigma_2(u_1,
v_1)-\sigma_2(u,
v)|^2&\leq&K_{\sigma_2}(|u_1-u|^2+|v_1-v|^2),\\
|\sigma_2(u)|^2&\leq&K_{\sigma_2}(1+|u|^2).
\end{eqnarray}
Moreover, there exist  constants $\alpha> 0$ and $C_{\sigma_2}$,
such that
\begin{eqnarray}
\nonumber v\cdot g(u, v, \xi)&\leq& \alpha|v|^2,\\ \sigma_2(u,
v)&\leq&C_{\sigma_2}.
\end{eqnarray}
\textbf{H3}: There exist $K_b$, $C_b$, $K_{\sigma_3}$, $C_{\sigma_3}$ such
that for any $\xi$, $\xi_1$, $\eta$, $\eta_1$,
\begin{eqnarray}
\nonumber|b(\xi, \eta)-b(\xi_1,
\eta_1)|^2&\leq&K_b(|\xi-\xi_1|^2+|\eta-\eta_1|^2),\\\nonumber
|b(\xi, \eta)|^2&\leq&K_b(1+|\xi|^2+|\eta|^2),\\\nonumber
|\sigma_3(\xi)-\sigma_3(\xi_1)|^2&\leq&K_{\sigma_3}|\xi-\xi_1|^2,\\\nonumber
|\sigma_3(\xi)|^2&\leq&K_{\sigma_3}(1+|\xi|^2),\\\nonumber |b(\xi,
\eta)|&\leq&C_b,\\ |\sigma_3(\xi)|&\leq&C_{\sigma_3}.
\end{eqnarray}
Furthermore, there exists a constant $\beta>0$, such that
\begin{eqnarray}
\xi\cdot b(\xi, \eta)\leq \beta(1+|\xi|^2).
\end{eqnarray}
\textbf{H4}: There exist $K_B$, $C_B$, $K_{\sigma_4}$, $C_{\sigma_4}$ such
that for any $\xi$, $\xi_1$, $\eta$, $\eta_1$,
\begin{eqnarray}
\nonumber|B(\xi, \eta)-B(\xi_1,
\eta_1)|^2&\leq&K_B(|\xi-\xi_1|^2+|\eta-\eta_1|^2),\\\nonumber
|B(\xi, \eta)|^2&\leq&K_B(1+|\xi|^2+|\eta|^2),\\\nonumber
|\sigma_4(\xi, \eta)-\sigma_4(\xi_1,
\eta_1)|^2&\leq&K_{\sigma_4}(|\xi-\xi_1|^2+|\eta-\eta_1|^2),\\\nonumber
|\sigma_4(\xi,
\eta)|^2&\leq&K_{\sigma_4}(1+|\xi|^2+|\eta|^2),\\\nonumber |B(\xi,
\eta)|&\leq&C_B,\\ |\sigma_4(\xi, \eta)|&\leq&C_{\sigma_4}.
\end{eqnarray}
\textbf{H5}: $2\lambda_1+2\alpha-K_{\sigma_2}>0$, where $\lambda_1$ is the
first
eigenvalue of the operator $-\Delta$.\\

Let $H$ be the Hilbert space $L^2(D)$, equipped with inner products
$(\cdot , \cdot)_H$, and norm $\|\cdot\|=(\cdot,
\cdot)_H^{\frac{1}{2}}$. Define the operator $A=\Delta$ with zero
Dirichlet boundary condition. Let $\{e_k(x)\}_{k\geq 1}$ be the
complete  orthogonal system of eigenfunctions in $H$ such that, for
$k=1,2\cdots$,
\begin{equation}\label{eigenfunction}
-\Delta e_k=\lambda_ke_k,\;\;e_k|_{\partial D}=0 ,
\end{equation}
with $0<\lambda_1\leq\lambda_2\leq\cdots\lambda_k\leq\cdots$. It is
well known that the semigroup $\{G_t\}_{t\geq0}$ generated by
$\Delta$ can be defined by,
\begin{equation*}
(G_th)(\varsigma)=\int_DG(\varsigma,\zeta,t)h(\zeta)d\zeta ,
\end{equation*}
for any $h(\varsigma)\in H$, where
$G(\varsigma,\zeta,t)=\sum\limits_{k=1}^\infty
e^{-\alpha_kt}e_k(\varsigma)e_k(\zeta).$ It is clear that
$\|G_th\|\leq\|h\|$, thus $\{G_t\}_{t\geq0}$ is a contraction semigroup. Let $V$
be the Sobolev space $H_0^1$ of order $1$ with Dirichlet boundary
conditions, which is densely and continuously injected in the
Hilbert space $H$. $V$, $H$ and $V^\star$ satisfies a Gelfand triple
$$V\subset H\subset V^\star,$$
and
$$\Delta: V \rightarrow V^\star.$$
With the Poincare inequality, we have
\begin{eqnarray}\label{Poincare}
\langle\Delta v, v\rangle=-\|\nabla v\|^2\leq-\lambda\|v\|^2,
\end{eqnarray}
where $\langle\cdot,\cdot\rangle$ denotes the dual pairs of
$(V^\star, V).$ \\
Under the assumptions, the macroscopic fast equation has a unique
stationary solution, with distribution $\mu^u$ independent of
$\eps$, and the average is
\begin{eqnarray}
\bar f(u, \xi)&=& \int_H f(u, v, \xi)\,\mu^u(dv), \quad u\in H,
\quad \xi\in \mathbb{R}.
\end{eqnarray}
Then we deal with the following macroscopic effective system
\begin{eqnarray}
\bar u_t &=& \bar u_{xx} + \bar f(\bar u,\bar \xi) + \sigma_1(\bar
u)\dot W_t^1, \quad \bar u(x, 0)=u_0(x).
\end{eqnarray}
Moreover, an averaged microscopic effective model for $\xi$ is defined   as in the last section.

 Now we consider the following effective system for the original microscopic-macroscopic system \eqref{macro-micro}:
\begin{eqnarray}\label{effective-new}
\begin{cases}
  \bar u_t  = \bar u_{xx} + \bar f(\bar u,\bar \xi) + \sigma_1(\bar
u)\dot W_t^1, \quad \bar u(x, 0)=u_0(x),  \\
 \dot{\bar{\xi}}  = \bar{b}(\bar\xi) + \sigma_3(\bar{\xi}) \dot{W}_t^3, \quad
 \bar{\xi}(0)=x_0.
\end{cases}
\end{eqnarray}
We have already known above that $\xi$ converges to $\bar{\xi}$ in probability, uniformly on bounded time intervals.
Our goal in this section is to show that $u$ converges to $\bar{u}$ in some probabilistic sense.

\bigskip


The well-posedness for both     systems
\eqref{macro-micro} and  \eqref{effective-new} is verified as in \cite{DaPrato}.

\begin{definition}
(Mild solution). For fixed $\xi$, an $H\times H-valued$ predictable
process $(u(t), v(t))$ is called a \emph{mild solution} of the first
two components of Eq. \eqref{macro-micro} if for any $t\in [0, T],$
\begin{equation}
\begin{cases}
u(t)=G_tu_0+\int_0^tG_{t-s}f(u(s),
v(s),\xi)ds+\int_0^tG_{t-s}\sigma_1(u(s))dW^1_s,\\
v(t)=G_t^\eps v_0+\frac{1}{\eps}\int_0^tG_{t-s}^\eps g(v(s), v(s),
\xi)ds+\frac{1}{\sqrt{\eps}}\int_0^tG_{t-s}^\eps\sigma_2(u(s),
v(s))dW^2_s,
\end{cases}
\end{equation}
where $\{G_{t}^\epsilon\}_{t\geq0}$ denote the semigroup generated
by differential operator $\frac{\Delta}{\epsilon}.$
\end{definition}

\begin{definition}
(Strong solution).  For fixed $\xi$, a $V\times V-valued$
predictable process $(u(t), v(t))$ is called a strong solution of
the first two components  of Eq. \eqref{macro-micro} if, for any
$\varphi\in V$,
\begin{equation}
\begin{cases}
\big(u(t), \varphi\big)_H=\big(u_0, \varphi\big)_H+\int_0^t\langle
\Delta u(s)^\epsilon, \varphi\rangle ds+\int_0^t\big(f(u(s),
v(s), \xi), \varphi\big)_Hds\\
\qquad\qquad\quad\,\,+\int_0^t\big(\sigma_1(u(s)),
\varphi\big)_HdW^1_s,\\
\big(v(t), \varphi\big)_H=\big(v_0,
\varphi\big)_H+\frac{1}{\eps}\int_0^t\langle \Delta v(s),
\varphi\rangle ds+\frac{1}{\epsilon}\int_0^t\big(g(u(s),
v(s), \xi), \varphi\big)_Hds\\
\qquad\qquad\quad\,\,+\frac{1}{\sqrt{\epsilon}}\int_0^t\big(\sigma_2(u(s),
v(s)), \varphi\big)_HdW^2_s,
\end{cases}
\end{equation}
hold for any $t\in [0, T] \,\,\, a.s.$.
\end{definition}


Under the assumptions we listed, for any fixed $u_0\in H$ and any
$v_0\in H$, the  first two equations of the Eq. \eqref{macro-micro}
has a unique strong solution (also a mild solution). Moreover, the
following energy identities hold ( \cite{Krylov} or \cite{Chow}):
\begin{eqnarray}
\nonumber\|u(t)\|^2&=&\|u_0\|^2+2\int_0^t\langle \Delta u(s),
u(s)\rangle ds+2\int_0^t\big(f(u(s),
v(s), \xi), u(s)\big)_Hds\\
\label{Energy-slow}{}&&+2\int_0^t\big(\sigma_1(u(s),
u(s)\big)_HdW^1_s+\int_0^t\|\sigma_1(u(s))\|^2ds,
\end{eqnarray}
and
\begin{eqnarray}
\nonumber\|v(t)\|^2&=&\|v_0\|^2+\frac{2}{\epsilon}\int_0^t\langle
\Delta v(s), v(s)\rangle ds+\frac{2}{\epsilon}\int_0^t\big(g(u(s),
v(s), \xi), v(s)\big)_Hds\\
\nonumber{}&&+\frac{2}{\sqrt{\epsilon}}\int_0^t\big(\sigma_2(u(s),
v(s)), v(s)\big)_HdW^2_s+\frac{1}{\epsilon}\int_0^t\|\sigma_2(u(s),
v(s))\|^2ds.\\
{}\label{Energy-fast}
\end{eqnarray}

Similar to the case in section 2, for fixed $u_0\in H$,
$\xi\in\mathbb{R}$ we introduce a fast motion with frozen slow
component
\begin{eqnarray}\label{fast}
\begin{cases}
dv(t)=[v_{xx}(t)+g(u_0, v(t), \xi)]dt+\sigma_2(u_0,
v(t))dW^2_t,\\
v(x, 0)=v_0(x),\quad x\in[0, 1],\\
v(0, t)=v(1, t)=0, \quad t\in[0, T].
\end{cases}
\end{eqnarray}
 Under the assumptions, for any fixed $u_0\in H$ and any $v_0\in
H$, the  Eq. \eqref{fast} has a unique strong solution (also a mild
solution), which will be denoted by $v^{u_0, v_0}(t)$. By energy
equality similar to \eqref{Energy-fast}, one can get
\begin{eqnarray*}
\mathbb{E}\|v^{u_0, v_0}(t)\|^2\leq
\|v_0\|^2-2(\lambda_1+\alpha)\int_0^t\mathbb{E}\|v^{u_0,
v_0}(s)\|^2ds+Ct.
\end{eqnarray*}
By the Gronwall's inequality again we have
\begin{equation*}
\mathbb{E}\|v^{u_0, v_0}(t)\|^2\leq
C\bigg(1+\|v_0\|^2e^{-2(\lambda_1+\alpha)t}\bigg).
\end{equation*}
Let $v^{u_0, v'_0}(t)$ be the solution of Eq. \eqref{fast} with
initial value $v(0)=v'_0$. With the aid of energy equality similar
to \eqref{Energy-fast}, we get that
\begin{eqnarray*}
\mathbb{E}\|v^{u_0, v_0}(t)-v^{u_0,
v'_0}(t)\|^2&=&\|v_0-v'_0\|^2+2\mathbb{E}\int_0^t\langle A(v^{u_0,
v_0}(s)-v^{u_0, v'_0}(s)), v^{u_0, v_0}(s)-v^{u_0, v'_0}(s)\rangle \,ds\\
{}&&+2\mathbb{E}\int_0^t\big(g(u_0, v^{u_0, v_0}(s), \xi)-g(u_0,
v^{u_0, v'_0}(s), \xi),
v^{u_0, v_0}(s)-v^{u_0, v'_0}(s)\big)_Hds\\
{}&&+\mathbb{E}\int_0^t\|\sigma_2(u_0, v^{u_0,
v_0}(s))-\sigma_2(u_0, v^{u_0, v'_0}(s))\|^2ds\\
&\leq&\|v_0-v'_0\|^2-(2\alpha_1+2\beta-C_{\sigma_2})\int_0^t\mathbb{E}\|v^{u_0,
v_0}(s)-v^{u_0, v'_0}(s)\|^2\,ds.
\end{eqnarray*}
Hence
\begin{equation}\label{Inital-fast-eqn}
\mathbb{E}\|v^{u_0, v_0}(t)-v^{u_0, v'_0}(t)\|^2\leq
\|v_0-v'_0\|^2e^{-\kappa t},
\end{equation}
where $\kappa=2\alpha_1+2\beta-C_{\sigma_2}>0.$\\

For any $u\in H$ denote by $P_t^u$ the Markov semigroup associated
to Eq. \eqref{macro-micro} defined by
$$P_t^uf(z)=\mathbb{E}f(V_t^{u, z}), \quad t\geq0,\quad z\in H,$$
for any $f\in\mathcal {B}_b(H)$ the space of bounded functions on
$H$. We also recall a probability $\mu^u$ on $H$ is called that a
invariant measure for $(P_t^u)_{t\geq 0}$ if
\begin{equation*}
\int_HP_t^uf d\mu^u=\int_Hf d\mu^u, \quad t\geq 0,
\end{equation*}
for any bounded function $f \in \mathcal {B}_b(H).$ As in
\cite{Cerrai1}, it is possible to show there exists an unique
invariant measure $\mu^u$ for the semigroup $P_t^u$ which satisfies
\begin{equation}
\int_H\|z\|\mu^u(dz)\leq C(1+\|u\|).
\end{equation}
Furthermore, according to Lipschitz assumption on $f$ and
\eqref{Inital-fast-eqn} we have
\begin{eqnarray}
&&\!\!\!\!\!\!\!\!\!\!\nonumber\Big\|\mathbb{E}f(u, V_t^{u, v},
\xi)-\int_Hf(u,
z, \xi)\mu^u(dz)\Big\| \\
\nonumber&=&\Big\|\int_H[\mathbb{E}f(u, V_t^{u,
v}, \xi)-\mathbb{E}f(u, V_t^{u, z}, \xi)]\mu^u(dz)\Big\|\\
\nonumber&\leq&C\int_H\mathbb{E}\|V_t^{u, v}-V_t^{u,
z}\|\mu^u(dz)\\
\nonumber&\leq&Ce^{-\frac{\kappa}{2}t}\int_H\|v-z\|\mu^u(dz)\\
\nonumber&\leq&Ce^{-\frac{\kappa}{2}t}\Big[\|v\|+\int_H\|z\|\mu^u(dz)\Big]\\
\label{Averaging-property}&\leq&Ce^{-\frac{\kappa}{2}t}\Big[1+\|u\|+\|v\|\Big].
\end{eqnarray}

The following arguments follow \cite{Fu}.
First we have some mean square uniform estimates on $u$, $v$, and
$\xi$.
\begin{lemma}There exists a
constant $C_T> 0$ such that
\begin{equation}\label{Slow-estimate}
\sup\limits_{0\leq t\leq T}\mathbb{E}| \xi|^2\leq C_T.
\end{equation}
\begin{proof}
For the slow equation $\xi$ of the microscopic system, multiplying each side
with $2\xi$, we get
$$\frac{d}{dt}|\xi|^2=2\xi\cdot b(\xi,
\eta)+2\xi\cdot\sigma_3(\xi)\dot W_t^3.$$
After integrating and taking expectation on both sides, we get
\begin{eqnarray*}
\mathbb{E}|\xi|^2&=&x^2+2\mathbb{E}\int_0^t \xi\cdot b(\xi,
\eta)\,ds+2\mathbb{E}\int_0^t \xi\cdot\sigma_3(\xi)dW_s^3\\
{}&\leq& x^2+2\beta t+2\beta\int_0^t \mathbb{E}|\xi|^2\,ds.
\end{eqnarray*}
 Thanks to the
Gronwall's inequality, we finally have
\begin{eqnarray*}
\mathbb{E}|\xi|^2&\leq& x^2+2\beta t+2\beta\int_0^t(x^2+2\beta s)e^{2\beta(t-s)}\,ds\\
{}&=&(1+x^2)e^{2\beta t}-1\\
{}&\leq& C_T.
\end{eqnarray*}
\end{proof}
\end{lemma}

\begin{lemma}  There exists a
constant $C> 0$ such that
\begin{equation}\label{Fast-estimate}
\sup\limits_{0\leq t\leq T}\mathbb{E}\| v\|^2\leq C.
\end{equation}
\begin{proof}
Due to energy identity \eqref{Energy-fast}, coercivity
\eqref{Poincare}, the assumption H2 and the Gronwall's inequality, we
obtain the desired result.\\
\end{proof}
\end{lemma}

\begin{lemma}There exists a
constant $C_T> 0$ such that
\begin{equation}\label{Slow-estimate2}
\sup\limits_{0\leq t\leq T}\mathbb{E}\| u(t)\|^2\leq C_T.
\end{equation}
\begin{proof}
Applying energy identity \eqref{Energy-slow},  with the aid of
\eqref{Poincare} and the above two lemmas , we get
\begin{eqnarray*}
\mathbb{E}\|u(t)\|^2&=&\|u_0\|^2+\mathbb{E}\int_0^t\langle \Delta
u(s), u(s)\rangle ds+\mathbb{E}\int_0^t\Big(f(u(s), v(s), \xi),
u(s)\Big)_Hds+\mathbb{E}\int_0^t\|\sigma_1(u(s)
)\|^2ds\\
{}&\leq&\|u_0\|^2+C\int_0^t\mathbb{E}\|u(s)\|^2ds
+C\int_0^t\mathbb{E}(1+\|u(s)\|^2+\|v(s)\|^2+\|\xi\|^2)ds\\
{}&\leq&\|u_0\|^2+C_T\int_0^t\mathbb{E}\|u(s)\|^2ds +Ct.
\end{eqnarray*}
The Gronwall's inequality  yields the desired estimation.\\
\end{proof}
\end{lemma}

\begin{lemma}\label{lemma-Holder}
For any $h\in (0, 1)$ and $\gamma\in (0, \frac{1}{2})$, there exists a
constant $C_\gamma>0$ such that
\begin{equation}\label{Holder-continuity}
\mathbb{E}\|u(t+h)-u(t)\|^2\leq C_\gamma h^\gamma.
\end{equation}
\begin{proof}
In the mild sense
\begin{eqnarray}
\nonumber
u(t+h)-u(t)&=&[G_{t+h}u_0-G_tu_0]+\int_t^{t+h}G_{t+h-s}f(u(s), v(s),
\xi)\,ds\\\nonumber
{}&&+\int_t^{t+h}G_{t+h-s}\sigma_1(X_s^\epsilon)dW^1_s\\\nonumber
{}&&+\int_0^t[G_{t+h-s}f(u(s), v(s), \xi)-G_{t-s}f(u(s), v(s),
\xi)]\,ds\\\nonumber
{}&&+\int_0^t[G_{t+h-s}\sigma_1(u(s))-G_{t-s}\sigma(u(s))]dW^1_s\\
{}\label{I_1---I_5}&=:&\sum\limits_{i=1}^{5}I_i.
\end{eqnarray}
By the property of semigroup $G_t$ (see \cite{pazy}), we   have the
estimate of $I_1$,
\begin{eqnarray}
\|I_1\|^2
\label{I-1}&\leq&h^2\|\Delta u_0\|^2.
\end{eqnarray}
By the H\"{o}lder inequality and the bounded property of $f$,
we   deduce that
\begin{eqnarray}\nonumber
\mathbb{E}\|I_2\|^2&\leq& h\mathbb{E}\int_t^{t+h}\|G_{t+h-s}f(u(s),
v(s), \xi)\|^2ds\\\nonumber &\leq&Ch\int_t^{t+h}\mathbb{E}\|f(u(s),
v(s), \xi)\|^2\,ds\\
\label{I-2}&\leq&Ch^2.
\end{eqnarray}
Using the It\^{o} isometry and H\"{o}lder inequality, it yields
\begin{eqnarray}\nonumber
\mathbb{E}\|I_3\|^2&=&\mathbb{E}\int_t^{t+h}\|G_{t+h-s}\sigma_1(u(s))\|^2ds\\
\nonumber&\leq&C\int_t^{t+h}\mathbb{E}[1+\|u(s)\|^2]\,ds\\
\label{I-3}&\leq&C_T h.
\end{eqnarray}
Moreover
\begin{eqnarray}\nonumber
\mathbb{E}\|I_4\|^2 
\label{I-4}&\leq&C_\gamma h^\gamma,
\end{eqnarray}
and
\begin{eqnarray}\nonumber
\mathbb{E}\|I_5\|^2 
\label{I-5}&\leq&C_{T,\gamma} h^\gamma,
\end{eqnarray}
are obtained in the same way as those in \cite{Fu}, where $f(u, v,
\xi)$ and $f(u, v)$ are both bounded. \\
As a result of \eqref{I_1---I_5}---\eqref{I-5}, we obtain inequality
\eqref{Holder-continuity}.
\end{proof}
\end{lemma}

Next, we introduce an auxiliary process $\left( \hat u(t), \hat
v(t)\right)\in H\times H$. Fix a positive number $\delta$ and do a
partition of time interval $[0, T]$ of size $\delta$. We construct a
process $\hat v(t)$ by means of the equations
\begin{eqnarray}
\nonumber\hat
v(t)=v(k\delta)&+&\frac{1}{\eps}\int_{k\delta}^t\Delta\hat v(s)
+\frac{1}{\eps}\int_{k\delta}^tg(u(k\delta),
\hat v(s), \xi(s))\,ds\\
\label{Auxiliary-fast}&+&\frac{1}{\sqrt{\eps}}\int_{k\delta}^t\sigma_2(u(k\delta),
\hat v(s))dW^2_s,
\end{eqnarray}
for $t\in \left[k\delta, \min\big((k+1)\delta, T\big)\right )$.\\
Also define the process $\hat u(t)$ by linear equation with additive
noise
\begin{eqnarray}
\nonumber\hat u(t)=u_0&+&\int_0^t\Delta\hat u(s)\,
ds+\int_0^tf(u([s/\delta]\delta), \hat v(s), \xi(s))\,ds\\
\label{Auxiliary-slow}{}&&+\int_0^t\sigma_1(u(s))\,dW^1_s,
\end{eqnarray}
for $t\in [0, T]$.\\

 Similar to  the mean square uniform estimates
on
 $v$, we have
\begin{lemma}  There exists a
constant $C> 0$ such that
\begin{equation}
\sup\limits_{0\leq t\leq T}\mathbb{E}\|\hat v\|^2\leq C.
\end{equation}
\end{lemma}

We now are ready to establish mean-square
convergence of the auxiliary processes $\hat v(t)$ and $\hat u(t)$
to the fast solution process $v(t)$ and slow $u(t)$,  respectively.
\begin{lemma}\label{Diff-fast-auxil}
For any $\gamma\in (0, \frac{1}{2})$, there exist constants $C_{T,
\gamma}>0$  such that $$\sup\limits_{0\leq t\leq
T}\mathbb{E}\|v(t)-\hat v(t)\|^2\leq C_{T,
\gamma}\frac{\delta^{1+\gamma}}{\eps}e^{\frac{C\delta}{\eps}}.$$
\begin{proof}
For $t\in [0, T]$ with $t\in [k\delta, (k+1)\delta)$, by energy
identity \eqref{Energy-fast} ,  \eqref{Poincare} and the Lipschitz
condition of $g(u, v, \xi)$  that
$$|g(u(s), v(s), \xi)-g(u(k\delta), \hat v(s), \xi)|^2\leq K_g (\|u(s)-u(k\delta)\|^2+\|v(s)-\hat v(s)\|^2),$$
 we get the desired
result.\\
\end{proof}
\end{lemma}

The next lemma is by the same argument with the help of
$$|f(u, v, \xi)-f(u([t/\delta]\delta), \hat v, \xi)|^2\leq K_f (\|u-u([t/\delta]\delta)\|^2+\|v-\hat v\|^2).$$
\begin{lemma}\label{Slow-auxil-diff}
For any $\gamma\in (0, \frac{1}{2})$, there exists constant $C_{T,
\gamma}>0$ such that $$\sup\limits_{0\leq t\leq
T}\mathbb{E}\|u(t)-\hat u(t)\|^2\leq C_{T,
\gamma}(\delta^\gamma+\frac{\delta^{1+\gamma}}{\eps}e^{\frac{C\delta}{\eps}}).$$\\
\end{lemma}

In the following we prove the averaging principle that the slow
component process $u(t)$ converges  in mean-square sense to an
effective dynamics equation as follows
\begin{equation}\label{Averaing-eqn}
\begin{cases}
d\bar u(t)=\Delta\bar u(t)dt+\bar f(\bar u(t), \bar\xi)dt+\sigma_1(\bar u(t))dW^1_t,\\
\bar u(x,0)=u_0(x).
\end{cases}
\end{equation}

The following lemma formulates mean-square convergence of the auxiliary
process $\hat u(t)$ to the averaged solution process $\bar u(t)$.
\begin{lemma}\label{Main-lemma}
For any $\gamma\in (0, \frac{1}{2})$, there exist constants $C_{T,
\gamma}>0$ such that
\begin{equation*}
\mathbb{E}\|\hat u(t)-\bar u(t)\|^2\leq C_{T,
\gamma}(\delta^\gamma+\frac{\eps}{\delta}+\frac{\delta^{1+\gamma}}{\eps}e^{\frac{C\delta}{\eps}}).
\end{equation*}

\begin{proof}
In the mild sense, we have
\begin{eqnarray*}
\hat u(t)-\bar u(t)&=&\int_0^tG_{t-s}[f(u([s/\delta]\delta), \hat
v(s), \xi)-\bar{f}(u(s), \xi)]\,ds
+\int_0^tG_{t-s}[\bar{f}(u(s) ,\xi)-\bar{f}(\hat u(s), \xi)]\,ds\\
{}&&+\int_0^tG_{t-s}[\bar{f}(\hat u(s), \xi)-\bar{f}(\bar u(s),
\bar\xi)]\,ds+\int_0^tG_{t-s}
[\sigma_1(u(s))-\sigma_1(\hat u(s))]\,dW^1_s\\
{}&&+\int_0^tG_{t-s}[\sigma_1(\hat u(s))-\sigma_1(\bar u(s))]\,dW^1_s\\
{}&:=&\sum\limits_{i=1}^5J_i(t).
\end{eqnarray*}
In view of the H\"{o}lder inequality, the Lipschitz condition of $\bar
f(u, \xi)$ and contraction of the semigroup $G_t$, it follows from Lemma
\ref{Slow-auxil-diff} that
\begin{eqnarray}
\nonumber\mathbb{E}\|J_2(t)\|^2  
\label{J-2}&\leq&C_{T,
\gamma}(\delta^\gamma+\frac{\delta^{1+\gamma}}{\eps}e^{\frac{C\delta}{\eps}}).
\end{eqnarray}
For $J_3$ , because of the Lipschitz continuity of $\bar{f}$ we have
\begin{eqnarray}
\nonumber\mathbb{E}\|J_3(t)\|^2&\leq&
C_T\mathbb{E}\int_0^t\|\bar{f}(\hat u(s), \xi)-\bar{f}(\bar u(s), \bar\xi)\|^2\,ds\\
\label{J-3}&\leq&C_T\int_0^t\mathbb{E}(\|\hat u(s)-\bar
u(s)\|^2+|\xi-\bar\xi|^2)\,ds.
\end{eqnarray}
Furthermore, $J_4$, $J_5$ are estimated  using the   properties of $G_t$ and Lemma
\eqref{Slow-auxil-diff},
\begin{eqnarray}
\nonumber\mathbb{E}\|J_4(t)\|^2  
\label{J-4}&\leq&C_{T,
\gamma}(\delta^\gamma+\frac{\delta^{1+\gamma}}{\eps}e^{\frac{C\delta}{\eps}}),
\end{eqnarray}
\begin{eqnarray}
\nonumber\mathbb{E}\|J_5(t)\|^2  
\label{J-5}&\leq&C\int_0^t\mathbb{E}\|\hat u(s)-\bar u(s)\|^2\,ds.
\end{eqnarray}
For $\mathbb{E}\|J_1(t)\|^2$ with $t\in[k\delta, (k+1)\delta)$, we
  write
\begin{eqnarray}
\nonumber
J_1(t)&=&\sum\limits_{p=0}^{k-1}\int_{p\delta}^{(p+1)\delta}G_{t-s}[f(u(p
\delta), \hat v(s), \xi)-\bar{f}(u(p\delta), \xi)]\,ds\\
\nonumber{}&&+\sum\limits_{p=0}^{k-1}\int_{p\delta}^{(p+1)\delta}G_{t-s}[\bar{f}(u(p\delta), \xi)-\bar{f}(u(s), \xi)]\,ds\\
{}&&+\int_{k\delta}^tG_{t-s}[f(u(p\delta), v(s), \xi)-\bar{f}(u(s),
\xi)]\,ds\\\nonumber \label{J-1}&:=&J'_1(t)+J'_2(t)+J'_3(t).
\end{eqnarray}
Due to \eqref{Holder-continuity}, we conclude
\begin{eqnarray}
\nonumber\mathbb{E}\|J'_2(t)\|^2  
\label{J'2}&\leq&C_{T, \gamma}\delta^\gamma,
\end{eqnarray}
with $\gamma\in (0, \frac{1}{2}).$\\
 According to the mean square uniform estimates on $u$, $\hat v$, $\xi$ and the linear growth conditions of $f$ and
 $\bar f$, we get
\begin{eqnarray}
\nonumber\mathbb{E}\|J'_3(t)\|^2&=&\mathbb{E}\|\int_{k\delta}^tG_{t-s}[f(u(k\delta), \hat v(s), \xi)-\bar{f}(u(s), \xi)]\,ds\|^2\\
\nonumber
&\leq&\delta\mathbb{E}\int_{k\delta}^t\|f(u(k\delta), \hat v(s), \xi)-\bar{f}(u(s), \xi)\|^2\,ds\\
\nonumber&\leq&C\delta\int_{k\delta}^t\mathbb{E}[1+\|u(k\delta)\|^2+\|\hat v(s)\|^2+\|u(s)\|^2+|\xi|^2]\,ds\\
\label{J'3}&\leq&C_T\delta^2.
\end{eqnarray}
 The argument of the estimate of
\begin{eqnarray}\label{J'1}
\mathbb{E}\|J'_1\|^2\leq C_T\frac{\eps}{\delta},
\end{eqnarray}
is the same as that in \cite{Fu}, except that the coefficient $f$
has an extra parameter $\xi$,      which can be handled using
  \eqref{Averaging-property}  and the boundedness conditions for $f$.\\

 Combing \eqref{J'2}, \eqref{J'3} and \eqref{J'1} it yields
\begin{equation}\label{J_1}
\mathbb{E}\|J_1(t)\|^2\leq C_{T,
\gamma}\delta^\gamma+C_T\frac{\eps}{\delta}.
\end{equation}
Therefore,  combining together \eqref{J-2}---\eqref{J-5} and \eqref{J_1}
we obtain
\begin{equation}
\mathbb{E}\|\hat u(t)-\bar u(t)\|^2\leq C_{T,
\gamma}(\delta^\gamma+\frac{\eps}{\delta}+\frac{\delta^{1+\gamma}}{\eps}e^{\frac{C\delta}{\eps}}+\sup\limits_{0\leq
t\leq T}\mathbb{E}|\xi-\bar\xi|^2)+C_T\int_0^t\mathbb{E}\|\hat
u(s)-\bar u(s)\|^2\,ds
\end{equation}
and thus
\begin{equation*}
\mathbb{E}\|\hat u(t)-\bar u(t)\|^2\leq C_{T,
\gamma}(\delta^\gamma+\frac{\eps}{\delta}+\frac{\delta^{1+\gamma}}{\eps}e^{\frac{C\delta}{\eps}}+\sup\limits_{0\leq
t\leq T}\mathbb{E}|\xi-\bar\xi|^2).
\end{equation*}
This proves the lemma.
\end{proof}
\end{lemma}

Finally we have  the following theorem.

\begin{theorem}(Effective dynamical reduction II) \label{Theo}\\
Under the Hypotheses (H1)---(H5),   the system \eqref{effective-new} is an effective description of the  original system \eqref{macro-micro}.
That is,     for any $T>0$,
\begin{equation} \label{Averging-limit}
\lim\limits_{\epsilon\rightarrow 0}\sup\limits_{0\leq t\leq
T}\mathbb{E}\|u(t)-\bar u(t)\|^2\rightarrow 0.
\end{equation}
This says that $u$ converges to  $\bar{u}$ in mean-square, uniformly
on finite time intervals.
\end{theorem}
\begin{remark}
According to \eqref{Averging-limit} and by the Chebyshev
inequality, there is a direct consequence that $u$ converges to
$\bar{u}$ in probability.
\end{remark}

\begin{proof}
By Lemma \eqref{Slow-auxil-diff} and Lemma \eqref{Main-lemma} and
take $\delta=\epsilon[-\ln\epsilon]^{\frac{1}{2}}$, we have
\begin{eqnarray*}
\sup\limits_{0\leq t\leq T}\mathbb{E}\|u(t)-\bar
u(t)\|^2&\leq&\sup\limits_{0\leq t\leq T}\mathbb{E}\|u(t)-\bar
u(t)\|^2
+\sup\limits_{0\leq t\leq T}\mathbb{E}\|u(t)-\bar u(t)\|^2\\
{}&\leq&C_{T,
\gamma}(\delta^\gamma+\frac{\eps}{\delta}+\frac{\delta^{1+\gamma}}{\eps}e^{\frac{C\delta}{\eps}}+\sup\limits_{0\leq
t\leq T}\mathbb{E}|\xi-\bar\xi|^2)\\{}&\rightarrow& 0,
\end{eqnarray*}
 as $\eps\rightarrow 0$ .
\end{proof}

 \bigskip

\noindent {\bf Acknowledgements.} We would like to thank Xiaofan Li  
for helpful discussions.


\end{document}